\documentclass[10pt]{article}

\usepackage[T1]{fontenc}
\usepackage[utf8]{inputenc}
\usepackage{lmodern}
\usepackage[USenglish]{babel}
\usepackage[babel=true]{microtype}

\usepackage{graphicx}
\usepackage{latexsym,color}
\usepackage{amssymb}
\usepackage{amsfonts}
\usepackage{amsmath}
\usepackage{amsthm}
\usepackage[inline]{enumitem}
\usepackage{commath}
\usepackage{xcolor}
\usepackage{xparse}
\usepackage{xpatch}

\usepackage{hyperref}

\newtheorem{thm}{Theorem}[section]
\newtheorem{prop}[thm]{Proposition}
\newtheorem{cor}[thm]{Corollary}
\theoremstyle{definition}
\newtheorem{dfn}[thm]{Definition}
\newtheorem{lem}[thm]{Lemma}
\newtheorem{asm}[thm]{Assumption}
\newtheorem{exm}[thm]{Example}
\newtheorem{rem}[thm]{Remark}

\def\theequation{\thesection.\arabic{equation}}
\numberwithin{equation}{section}
\newcounter{proofstep}
\AtBeginEnvironment{proof}{\setcounter{proofstep}{0}\newcommand{\step}[1]{\refstepcounter{proofstep}\emph{Step \arabic{proofstep}. #1} }}

\newcommand{\cA}{\mathcal{A}}

\newcommand{\cC}{\mathcal{C}}

\newcommand{\cF}{\mathcal{F}}

\newcommand{\cI}{\mathcal{I}}

\newcommand{\cL}{\mathcal{L}}
\newcommand{\cM}{\mathcal{M}}

\newcommand{\cX}{\mathcal{X}}

\def \E{\mathbb{E}}

\def \N{\mathbb{N}}

\def \R{\mathbb{R}}

\def\1{{\bf{1}}}
\def\0{{\bf{0}}}

\def\reff#1{{\rm(\ref{#1})}}

\newcommand{\dfull}[4]{\sum_{j=1}^{\infty} c_j \langle #1,#2\rangle^#3#4}
\newcommand{\ds}[2][f_j]{\dfull{#2}{#1}{2}{}}
\NewDocumentCommand{\dsder}{ O{f_j} O{f_j} m }{\dfull{#3}{#1}{{}}{#2}}

\newcommand{\der}{D_m}

\newcommand{\dsup}[1]{d_{#1}}
\newcommand{\dist}{d}

\newcommand{\jet}[2]{J^{#1,#2}}
\newcommand{\cjet}[2]{\expandafter\overline\jet{#1}{#2}}
\newcommand{\test}{\Phi}

\newcommand{\law}{\cL}
\newcommand{\ca}{ca(\R)}

\newcommand{\controlset}{\mathcal{A}}
\newcommand{\diffop}{\mathcal{L}}

\newcommand{\invSet}{\mathcal{O}}
\newcommand{\invSetN}{\invSet_{\!N}}
\newcommand{\probSpace}{\mathcal{P}}
\newcommand{\mSpace}{\probSpace}
\newcommand{\expmoments}{\mathcal{M}}
\newcommand{\invMeas}[1]{\expmoments_{#1}}
\newcommand{\invMeasN}[1]{\expmoments_N^{#1}}
\newcommand{\invSetNCl}[1]{\expandafter\overline\invSetN^{#1}}
\newcommand{\invSetCl}{\overline\invSet{}}

\newcommand{\supdiff}[1]{\jet{1}{+}_{#1}}
\newcommand{\supdiffN}{\supdiff{\invSetN}}
\newcommand{\subdiff}[1]{\jet{1}{-}_{#1}}
\newcommand{\subdiffN}{\subdiff{\invSetN}}

\begin{document}

\title{
Viscosity solutions for
controlled McKean--Vlasov jump-diffusions
}

\author{Matteo Burzoni\thanks{Mathematical Institute, University of Oxford. Partly supported by the Hooke Research Fellowship from the University of Oxford.}\phantom{\footnotesize 1}\footnotemark[4]
    \and Vincenzo Ignazio\thanks{Department of Mathematics, ETH Z\"urich.}\phantom{\footnotesize 1}\footnotemark[4]
\and A. Max Reppen\thanks{Department of Operations Research and Financial Engineering, Princeton University. Supported by the Swiss National Science Foundation grant SNF 181815.}
\and H. Mete Soner\footnotemark[2]\phantom{\footnotesize 1}\thanks{Partly supported by the Swiss National Science Foundation grant SNF 200020\_172815.}\phantom{\footnotesize 1}\thanks{Partly supported by the ETH Foundation and the Swiss Finance Institute.}
}

\date{\today}

\maketitle
\markboth{Burzoni, Ignazio, Reppen \& Soner}{Controlled McKean--Vlasov jump-diffusions}

\begin{abstract}
We study a class of non linear
integro-differential equations on the 
Wasserstein space related to 
the optimal control of McKean--Vlasov jump-diffusions.
We develop an intrinsic notion of viscosity solutions that does not rely on the lifting to an Hilbert space
and prove a comparison theorem for these solutions.
We also show that the value function is the unique viscosity solution.
\end{abstract}

\renewcommand{\theequation}{\arabic{section}.\arabic{equation}}
\pagenumbering{arabic}

\section{Introduction}

The main goal of this paper
is to develop a viscosity theory for
integro-differential  equations on the 
Wasserstein space related to 
the optimal control of McKean--Vlasov jump-diffusions.
These control problems
 are motivated
by the mean field games theory
developed by Lasry \& Lions \cite{LL1,LL2,LL3} 
(see also the videos of the College de France
lectures of Lions \cite{Lions}) and
by Huang, Caines \& Malham{\'e} \cite{HCM1,HCM2,HCM3}.
Although the mean-field games and McKean--Vlasov
control problems are related, there are subtle differences between 
these problems,
and a thorough introduction is 
given by Carmona, Delarue \& Lachapelle \cite{CDL}.
Indeed, for both problems the master equations  share many common
properties as initially derived by Bensoussan, Freshe \& Yam
\cite{BFY1,BFY2,BFY3}.
We refer to the videos of Lions \cite{Lions}, the lecture notes of Cardaliguet \cite{cnotes}
and the exhaustive book of Carmona \& Delarue \cite{CD} for more 
information on both problems and also for further references.

The state space of these problems is 
the set of probability measures, and in most
applications the Wasserstein space of 
probability measures with finite second moments are used.
Since this space is not linear, one encounters 
some difficulties in differentiation
and Lions \cite{Lions}
observed that one can naturally lift functions defined on
the Wasserstein space to functions on an appropriate $\cL^2$ space,
which allows for standard differentiation
and more importantly an immediate
use of It\^{o}'s calculus.
This approach is then used 
by Cardaliguet, Delarue, Lasry \& Lions
\cite{CDLL} to obtain the regularity
of the solutions to the master equation
of a mean-field game.  This very strong regularity 
result  implies in particular for a classical interpretation 
of the master equations on the Wasserstein space.  On the other hand,
in the absence of such strong regularity,
one needs to develop the notion of viscosity solutions 
for McKean--Vlasov control problems.
Pham \& Wei \cite{PW,PW1} initiated this study
using Lions' lifting for controlled diffusion processes.
Bandini, Cosso, Fuhrman \& Pham \cite{BCFP1,BCFP2}
further developed this theory 
for the dynamic programming equations for the 
partially observed systems which also have the same structure.  
An important advantage of this
approach to viscosity theory, in addition to the Hilbert structure
of $\cL^2$, is its ability to
utilize the existing results for viscosity solutions
on Hilbert spaces \cite{li,FGS}.  An intrinsic 
approach to viscosity solutions without lifting 
could also have advantages and Wu \& Zhang \cite{zhang}
studies this approach for diffusion process
using the techniques developed for path-dependent
viscosity solution \cite{path1,path2}.

Our goal is to develop  a viscosity theory for
jump-diffusion processes.  For the standard control
problems, the corresponding 
dynamic programming equations 
contain a non-local integral terms
related to the infinitesimal generator of
the jump-Markov processes.  Still
these equations have maximum principle
and  a viscosity theory is appropriate.
Starting from \cite{soner1988,soner1989,soner1990,sayah}
definitions, stability and comparison results for nonlinear integro-differential
equations of this type have
been developed.  We refer to more recent paper
by Barles \& Imbert \cite{barles} for more information.

The jump terms in these equations introduce
several new aspects.  In particular,  
for the McKean--Vlasov control problems,
the operator appearing in the dynamic programming equations
do not act on the Lions derivative (i.e., the derivative in the $\cL^2$
space of the lifted function)
but rather on the standard (sometimes called linear) derivative 
on the Wasserstein space.
Indeed, when all functions are smooth, 
 it is immediate that
the Lions derivative is an $\cL^2$ function
and it is equal to the space derivative of the linear 
derivative (see  Section 5.4 in \cite{CD}).  For the diffusion problems,
only the space derivatives of 
the linear derivative appear in the 
dynamic programming equation and therefore 
one can simply replace them by the Lions derivative.
For the integro-differential equations however,
one needs to recover the linear
derivative from the Lions derivative 
even to state the equations.  
Unfortunately the required regularity (to 
immediately connect these two derivatives)
is not
readily available when one is working 
in the viscosity structure. 

We choose
to work directly on the Wasserstein space with the  
linear derivative
to develop an intrinsic theory.
Although this approach has several
advantages, the dynamic programming equations
on the Wasserstein space are not as well 
studied 
as the lifted equation on the $\cL^2$ spaces
and parts of the viscosity theory has to be 
revisited.
Indeed, we  first provide 
appropriate definitions of viscosity sub and super-solutions
for a class of integro-differential 
equations in this space.  
We then show that the value function is a viscosity
solution in this sense.  Several properties 
of the dynamics is used to construct 
the framework that is appropriate for this
problem. In particular, we consider the 
equation only on the subset of the 
measures that have exponential moments.

One of the main contributions of this paper
is a comparison result for the viscosity solutions 
on the Wasserstein space.  
An important ingredient of our approach is a
 distance like function $d$
given for two probability measures $\mu, \nu$ by,
\[
d(\mu,\nu)= \sum_{j=1}^\infty c_j \langle\mu-\nu,f_j\rangle^2,
\]
where the countable set $\{f_j\}_{j \in \N}$
is carefully constructed to have
several important invariance type properties.
In the 
standard doubling-variables argument,
we penalize the two
points using $d$.  Then the subtle properties of 
$f_j$ 
allow us to estimate its linear derivative of $d$
by itself.  

The paper is organized as follows.  We first introduce
a class of optimal control problems of McKean--Vlasov
type in the next section.  A guiding example
for this class is a model of technological innovation
\cite{kortum,achdou}.  We discuss this problem
in Section \ref{s.innovation}.  The natural state space 
for this study is the subset of the Wasserstein space
of measures with exponential moments and under mild
assumptions, the corresponding dynamical system 
lives in this space.  In Section \ref{s.sigma} we
define this space, prove its functional analytic properties 
and show its connection to the 
controlled dynamics. In Section \ref{s.test} we give the 
definition of a viscosity solution and in Section \ref{s.DPE_viscosity}
show that the
value function is a viscosity solution.  Section \ref{s.comparison} provides
the construction 
of the functions $f_j$ and the comparison result.  We prove several technical results
in the Appendix.
\vspace{10pt}

{\em{Notation.}}
For a random variable $X$, defined on a probability space 
$(\Omega,\cF,P)$, 
we denote by $\law(X)$ the distribution of $X$ under $P$.
We denote by $\probSpace(\R)$ the space of probability measures 
on $\R$ and by $\ca$ the linear space of countably additive measures.
For any $\mu\in\probSpace(\R)$ and for any integrable function $f:\R\to\R$, we 
use the standard compact notation $\langle\mu,f\rangle:=\int_{\R}f(x)\mu(\dif x)$.
If $f$ is smooth, $f^{(i)}$ denotes the $i$-th order derivative of $f$ with $f^{(0)}=f$.
We endow the space of probability measures $\probSpace(\R)$  with the weak$^*$ topology $\sigma(\probSpace(\R),\cC_b(\R))$, where $\cC_b(\R)$ is the space of continuous and bounded functions on $\R$.
We denote by $\mu_n\rightarrow\mu$ the $\sigma(\probSpace(\R),\cC_b(\R))$-convergence of $\mu_n$ to $\mu$, i.e.,
$\langle\mu_n,f\rangle$ converges to $\langle\mu,f\rangle$ for every $f \in \cC_b(\R)$.

\section{The optimal control problem and the assumption}
\label{s.control}

Let $(\Omega,\cF,(\cF_s)_{s\in[0,T]},P)$ 
be a given filtered probability space supporting the 
following class of controlled McKean--Vlasov 
stochastic differential equations (SDEs) with initial condition
 $\law(X^\alpha_t)= \mu\in\mSpace(\R)$ and
\begin{equation}\label{eq:SDE}
\dif X^\alpha_s=b(s,\law(X^\alpha_s),\alpha_s)\dif{s}
+\sigma(s,\law(X^\alpha_s),\alpha_s) \dif W_s +\dif J_s,
\quad s >t,
\end{equation}
where $J_s$ is a purely discontinuous process with controlled intensity $\lambda(s,\law(X^\alpha_s),\alpha_s)$ and jump size given by an independent random variable $\xi$ with distribution $\gamma\in\probSpace(\R)$.
The class of admissible controls $\controlset$ is 
the set of all measurable deterministic functions of time
 with values in a prescribed measurable space $A$.
The value function is then given by
\[
V(t,\mu):=\inf_{\alpha\in \controlset}\left[\int_t^T L(s,\law(X^\alpha_s),\alpha_s)\dif{s}
+G(\law(X^\alpha_T))\right],
\]
with given functions $L$ and $G$.
The optimal control problem consists of 
finding the value $V$ and a minimizer (if it exists).

We close this section
by stating a set of conditions  assumed
to hold throughout the paper and they will not always be stated
explicitly later on.

\begin{asm}\label{a.standing}
There exist constants $C_0,\kappa_0,\delta>0$ such that the coefficients $b,\sigma,\lambda,L:[0,T]\times\mSpace(\R)\times A\to\R$ satisfy:
\begin{enumerate}[label=\textnormal{\textbf{(H\arabic*)}},ref=(H\arabic*)]
\item \label{ass:bounded}
for any $\mu\in\mSpace(\R)$, $a\in A$, $s\in[0,T]$,
\[
|b(s,\mu,a)|+|\sigma(s,\mu,a)|+|\lambda(s,\mu,a)|\le C_0
\]
\item \label{ass:convergence}
there exists a finite set $\cI\subset\N$ such that for any $\mu,\mu'\in\mSpace(\R)$, $a\in A$, $t,s\in[0,T]$
\begin{align*}
|b(t,\mu,a)-b(s,\mu',a)|+|\sigma(t,\mu,a)-\sigma(s,\mu',a)|&\le\kappa_0\bigg(|t-s|+\sum_{i\in \cI}|\langle\mu - \mu',x^i\rangle|\bigg),\\
|\lambda(t,\mu,a)-\lambda(s,\mu',a)|&\le\kappa_0\bigg(|t-s|\sum_{i\in \cI}|\langle\mu - \mu',x^i\rangle|\bigg).
\end{align*}
\item \label{ass:state} $\gamma$ has $\delta$-\emph{exponential moment}:
\[
\int_{\R}\exp({\delta|x|})\gamma(\dif x)<\infty.
\]
\item \label{ass:cost}
    $L$ is of the form $L_1(t,\mu, a)+ L_2(a) \langle\mu,L_3(\cdot)\rangle$, where
    $L_1:[0,T]\times\mSpace(\R)\times A\to\R$ is continuous in $(t,\mu)$, 
    uniformly in $a$, $L_2 : A \to \R$ with $\sup_{a \in A} L_2(a) < \infty$, 
    and $L_3:\R\to\R$ satisfies $|L_3(x)x|\le C_0\exp(\delta |x|)$ for every $x\in\R$.  
    The terminal cost $G$ is continuous.
\end{enumerate}
\end{asm}

{\em{In what follows, the constants $C_0,\kappa_0,\delta>0$ are 
always as in the above assumption.}}

\section{State space and dynamic programming}
\label{s.state}

Since the Brownian motion has 
exponential moments, Assumption \ref{a.standing},
in particular \ref{ass:state}, 
implies that the solutions of the state equation \eqref{eq:SDE}
has also exponential moments.  Therefore it is natural
to study the optimal control problem
in $\invSet := [0,T) \times\expmoments$, where $\expmoments$
is the subset of probability measures with $\delta$-exponential moments,
i.e., 
\[
\mu\in\expmoments \quad
\Leftrightarrow
\quad
\langle\mu,\exp(\delta|\cdot|)\rangle = \int_\R 
\exp(\delta|x|) \mu(\dif x)<\infty,
\]
where $\delta$ is as in \ref{ass:state}.
Our first result is the well-posedness of the problem 
and its straightforward proof 
is given in Appendix \ref{s.solutionSDE}.
\begin{thm} Under Assumption \ref{a.standing}, the SDE \eqref{eq:SDE} has a unique solution 
for any $\alpha\in\controlset$. 
\end{thm}

The described McKean--Vlasov control problem 
is deterministic and therefore, it is classical that  
the dynamic programming principle (DPP)  holds \cite{FS}, 
\begin{equation}\label{eq:DPP}
V(t,\mu)= \inf_{\alpha\in \controlset} \left[\int_t^\theta L(s,\law(X^\alpha_s),\alpha_s)\dif{s} 
+V(u,\law(X^\alpha_u))\right], \quad
\forall \theta \in [t,T].
\end{equation}

We need several definitions to formally state the corresponding
dynamic programming equation.  

\begin{dfn}
\label{d.derviative}
For $\varphi:\mSpace(\R)\to\R$, when exists, \emph{the linear derivative} of $\varphi$ 
at $\mu \in \mSpace(\R)$ is a function 
$\der\varphi:\mSpace(\R)\times\R\to \R$ such that for every $\mu,\mu'\in\mSpace(\R)$,
\[
\varphi(\mu)-\varphi(\mu')=\int_0^1\int_{\R}\der\varphi(\lambda\mu+(1-\lambda)\mu',x)\ (\mu-\mu')(dx)  \dif\lambda.
\]
When $\varphi:[0,T]\times \mSpace(\R)\to\R$, with an abuse
of notation, we denote the linear derivative with respect to
the $\mu$-variable still by $\der\varphi:[0,T]\times
\mSpace(\R)\times\R\to \R$.
\end{dfn}

 This derivative was used by Fleming \& Viot \cite{FV} to
study a martingale problem  in populations dynamics.
Also  recently Cuchiero, Larsson \& Svaluto-Ferro \cite{CLS} 
provided several of its properties in the  context of polynomial diffusions.

\begin{rem}\label{rmk:derFrechet}
Consider the linear function $\varphi(\mu)=\langle\mu,f\rangle$
with some $f:\R\to\R$.  It is 
immediate that $\der \varphi(\mu,x)=f(x)$ for any $(\mu,x)\in \mSpace(\R)\times\R$.
Moreover, suppose that $\varphi:\ca\to\R$ is Frechet differentiable and such that $D\varphi:\ca\to\R$ can be represented as $D\varphi[\mu]=\langle\mu,f\rangle$, for some $f:\R\to\R$. Then $f=\der \varphi$, namely, $D\varphi[\mu]=\langle\mu,\der \varphi\rangle$.

By the chain rule,  the linear derivative of $\varphi(\mu)=F(\langle\mu,f\rangle)$
with some smooth function $F$ is equal to 
$\der \varphi(\mu,x)=F^\prime (\langle\mu,f\rangle)f(x)$.\qed
\end{rem}

For a given input function $v=v(t,\mu,x)$, the operator 
$\diffop^{a,\mu}_t$ acting on the $x$-variable is given by,
\begin{alignat*}{2}
    \diffop^{a,\mu}_t[v](x)&:={}&&b(t,\mu,a)\frac{\partial v}{\partial x}(t,\mu,x)  
    +\frac{1}{2}\sigma^2(t,\mu,a) \frac{\partial^2 v}{\partial x^2}(t,\mu,x)\\
    &&&\quad+\lambda(t,\mu,a)\int_{\R}(v(t,\mu,x+y)-v(t,\mu,x))\gamma(dy).
\end{alignat*}

Using the above definitions, 
classical considerations starting from \eqref{eq:DPP}
formally lead to the following 
dynamic programming equation:
\begin{equation}\label{eq:DPE}
-\partial_t V(t,\mu) + \sup_{a \in A}H^a(t,\mu,\der V)=0,
\end{equation}
where,
\[
H^a(t,\mu, v):=- L(t,\mu,a)-\langle\mu, \diffop^{a,\mu}_t[v]\rangle.
\]

Indeed, as in the finite-dimensional 
optimal control theory,
if the value function is smooth
and cylindrical (i.e., if 
$V$ has the form
\[
    V(t,\mu)=F(t,\langle\mu,f_1\rangle,\ldots,\langle\mu,f_n\rangle)
\]
for some smooth functions $F$ and $f_1,\ldots,f_n$,
then it is possible to derive \eqref{eq:DPE} rigorously.
Importantly, in this case, the classical It\^{o}'s Formula 
can be applied to  $V(u,\law(X^\alpha_u))=
F(u,\E[f_1(X^\alpha_u)],\ldots,\E[f_n(X^\alpha_u)])$ 
for any control $\alpha$ and time $u$.
However, this assumption on the value
function is not expected to hold and also is not needed.
In Section~\ref{s.DPE_viscosity}, 
we prove that the value function is the unique viscosity solution to \eqref{eq:DPE}
even when it is neither smooth nor cylindrical.

\section{A model of technological innovation}
\label{s.innovation}

We briefly present here an example of a McKean--Vlasov control 
problem where the underlying process is a jump-diffusion.
The controlled equations represent a model of knowledge
diffusion which appeared in the macroeconomic literature in the area of 
search-theoretic models of technological change, e.g.\ \cite{achdou}\footnote{We thank Rama Cont for bringing this paper to our attention.},\cite{kortum}.
With controls $\alpha=(\alpha^1,\alpha^2)$, a social planner 
aims at promoting technological innovation in the society by controlling the process
\begin{equation*}
\dif X^\alpha_s=b\big(\E[X^\alpha_s],\alpha^1_s\big)\dif{s}+\sigma \dif W_s + \dif J_s,
\end{equation*}
where $X_0\sim\mu_0$ and $J_s$ is a purely discontinuous process 
with controlled intensity $\lambda(\E[X^\alpha_s],\alpha^2_s)$ and jump size given by a non-negative independent random variable $\xi$ with distribution $\gamma$.
The value $\exp(X^\alpha_s)$ represents the efficiency of the production of a continuum of consumption goods (technological frontier),
and the initial (logarithmic) efficiency is represented by the distribution $\mu_0$.
The aim is to maximize the average efficiency of the production of goods in order to foster the growth of the economy:
\[
    \text{maximize!}\quad \E\left[\int_0^T (1 - \alpha^2_s) \exp(X^\alpha_s)-(\alpha^1_s)^2\dif{s}\right],
\]
where $\alpha=(\alpha^1,\alpha^2)$ is chosen from an
appropriate class of deterministic processes.

The social planner can promote innovation by issuing research funds (exercising the control $\alpha^1$). On the other hand, she can promote exchange of ideas by setting up meetings at a controlled Poisson rate.
Meetings have the effect of inducing a non-negative  jump in the technological frontier, according to a random variable with distribution $\gamma$.
The functions $\lambda$ and $b$ are bounded since meetings cannot happen too frequently and research funds have a limited impact on the technological frontier.
These functions also depends on the distribution of $X^\alpha$ through its mean.
This aspect represents the positive feedback effect of a productive economy: if the average productivity is higher, technological improvements and meetings happen spontaneously at a higher rate.
Finally, the random Brownian component incorporate fluctuations in the efficiency of the production due to external contingent factors.
This model satisfies Assumption \ref{a.standing} under some 
appropriate regularity conditions on the parameters and initial distribution.

We refer to \cite{achdou} for further
examples of problems where the 
controlled process is only a diffusion without jump terms.

\section{$\sigma$-compactness of the state space}
\label{s.sigma}

Recall that
$\invSet := [0,T)\times\expmoments$, and  $\expmoments$ 
is the set of probability measures $\mu$ satisfying 
$\langle\mu,\exp(\delta|\cdot|)\rangle<\infty$,
where $\delta$ is as in \ref{ass:state}.
We endow this space with the subspace topology induced by $\mSpace$, 
i.e., weak$^*$ convergence.
We use the product topology on 
$\invSetCl{}: = [0,T]\times\expmoments$,
and emphasize that $\invSetCl{}$ is not the topological closure of $\invSet$, but simply includes the final time.

The space $\invSet$
has a  suitable $\sigma$-compact structure 
which is compatible
with the McKean--Vlasov dynamics.
This representation of $\invSet$
 is instrumental to obtain uniform integrability of the viscosity test 
functions as well as some continuity properties of the Hamiltonian. We continue
by constructing this structure.

For $\delta$ as in \ref{ass:state}, set
$$
e_\delta(x):= \exp(\delta [\sqrt{x^2+1} -1]), \quad x \in \R.
$$
We note that  $e_\delta$ is twice continuously differentiable and
$$
\exp(\delta [|x|-1]) \le e_\delta(x) \le  \exp(\delta|x|) \leq e^\delta e_\delta(x),
\quad \forall x \in \R.
$$
 For  $N \in \N$
 and $C_0,\delta$
as in Assumption \ref{a.standing}, let
\[
 \invSetN := \big\{(t,\mu) \in[0,T)\times\mSpace(\R)\mid \langle \mu,e_\delta\rangle\le
 Ne^{K^*t } \big\},
 \]
where 
\begin{equation}
\label{eq:K}
K^*= K^*(C_0,\delta):=
\frac{\delta C_0}{2}(2+C_0+\delta C_0) +C_0\bigg(\int_{\R}e^{\delta |x|}\gamma(\dif x)-1\bigg).
\end{equation}
The exact definition
of $K^*$ is not important for the functional analytic properties
of $\invSetN$ but is used centrally  
in the next lemma to prove an
invariance property.
  
It is clear that  
$\invSet = [0,T)\times\expmoments = \cup_{N=1}^\infty \invSetN$
and $\invSetCl{} = \cup_{N=1}^\infty \invSetNCl{}$, where
\[
\invSetNCl{}  := \big\{(t,\mu)\in[0,T]\times\mSpace(\R)\mid 
\langle \mu,e_\delta\rangle\le
 Ne^{K^*t } \big\}.
 \]
 We also use the following notation for a constant  $b>0$,
 \[
 \invMeas{b} := \big\{\mu \in \mSpace(\R)\mid \langle \mu,e_\delta\rangle\le b \big\}.
 \]
 
The following lemma shows that for each $N$,
$\invSetN$ and thus also $\invSet$, remains invariant under the controlled dynamics \eqref{eq:SDE}
for any control.
In particular, this means that for any given initial law $\mu \in \invSetN$,
we may restrict the dynamic programming equation \eqref{eq:DPE} 
to $\invSetN$.

\begin{lem}\label{lem:exp_mom}
Under Assumption \ref{a.standing}, for any $N \in \N$,
the set $\invSetN$ is invariant for the SDE \eqref{eq:SDE}, namely,
\[
(t,\mu)\in \invSetN \Rightarrow (u,\law(X^{t,\mu,\alpha}_u))\in 
\invSetN \quad \forall(u,\alpha)\in [t,T]\times \controlset,
\]
where $(X^{t,\mu,\alpha}_u)_{u\in[t,T]}$ is the solution to \eqref{eq:SDE} with initial condition 
$\law(X^{t,\mu,\alpha}_t)=\mu$.
\end{lem}
\begin{proof}
Set $\varphi(x):=\sqrt{x^2+1} -1$ so that
$$
e_\delta(x)=  e^{\delta \varphi(x)}, \quad x \in \R.
$$
It is clear that $\varphi$ is twice continuously differentiable and
both  $|\varphi'|$ and $\varphi'' >0$ 
are bounded by $1$.

Fix $(t,\mu)\in\invSetN$ and $\alpha \in \controlset$.
For $u\in[t,T]$,
set $Y_u:=e^{\delta \varphi(X_u)}$,
$\mu_u:=\law(X_u)$,
where $X_u:=X^{t,\mu,\alpha}_u$.
In particular, $\mu_t = \mu$ for any control $\alpha$
and by  It\^{o}'s Formula,
\begin{alignat*}{2}
    Y_u&={}&& Y_t+\int_t^u b(s,\mu_s,\alpha_s)
     \delta \varphi'(X_s) Y_s\dif{s} \\
     &&&+\frac{1}{2}\int_t^u\sigma^2(s,\mu_s,\alpha_s)
     \big[\delta\varphi''(X_s)+\delta^2(\varphi'(X_s))^2\big]Y_s\dif{s}\\
     &&&+\int_t^u\sigma(s,\mu_s,\alpha_s)
     \delta\varphi'(X_s)Y_s \dif W_s+\sum_{0\le s\le t}\Delta Y_s.
\end{alignat*}
In view of Assumption \ref{ass:bounded}, the stochastic integral in the above formula is a local martingale.
We  take expectation on both side up to a localizing sequence of stopping times 
$\{\tau_n\}_n$.
We also use Assumption \ref{ass:bounded}, to estimate
the expectation of the second and third term of the above sum 
are bounded by
\[
C_1\E\bigg[\int_t^{u} Y_{s\wedge\tau_n}\dif{s}\bigg],
\]
where $C_1:=\frac{\delta C_0}{2}(2+C_0+\delta C_0)$ and $C_0$ is as in Assumption \ref{ass:bounded}.

We next estimate  $e_J:=\E[\sum_{0\le s\le t\wedge\tau_n}\Delta Y_s]$.
First observe that for any $x,y \in \R$,
$|\varphi(y+x) -\varphi(y)| \le |x|$.
We then estimate by using Assumption \ref{ass:state}, 
\begin{align*}
e_J&=\E\int_t^{u\wedge\tau_n}\lambda(s,\mu_s,\alpha_s) 
\int_\R e^{\delta\varphi(X_{s\wedge\tau_n}+x)}
-e^{\delta\varphi(X_{s\wedge\tau_n})}\gamma(\dif x)\dif{s}\\
&\le C_0\E\int_t^{u\wedge\tau_n}Y_{s\wedge\tau_n}\int_\R 
\big(e^{\delta |x|}-1\big)\gamma(\dif x)\dif{s}\\
&\le C_2\E\int_t^{u}Y_{s\wedge\tau_n}\dif{s},
\end{align*}
where $C_2:=C_0 \big(\int_{\R}e^{\delta|x|}\gamma(\dif x)-1\big)$.
These and Fubini's Theorem imply that
\[
\E[Y_{u\wedge\tau_n}] \le \E[Y_t]+K^*  \int_t^{u} \E\big[Y_{s\wedge\tau_n}\big]\dif{s},
\]
where $K^*$ is as in \eqref{eq:K}. 
By Gronwall's Lemma and Fatou's Lemma, 
\[
\E[Y_u]\le e^{K^*(u-t)} \E[Y_t]
=e^{K^*(u-t)}
\ \langle \mu, e^{\delta \varphi}\rangle= 
e^{K^*(u-t)}\langle \mu, e_\delta \rangle.
\]
As $(t,\mu) \in \invSetN$, $\langle \mu, e_\delta \rangle \le N e^{K^*t}$.  Hence,
$$
\E[Y_u]\le e^{K^*(u-t)}\langle \mu, e_\delta \rangle
\le e^{K^*u}.
$$
\end{proof}

We provide the proof of the following simple result for
completeness.

\begin{lem}\label{lem:compactness}
For $N \in \N$,
$\invSetNCl{}$ is a compact subset of $[0,T]\times\mSpace(\R)$.
\end{lem}
\begin{proof}
For and $b>0$ and $R$ sufficiently large,
\[
\sup_{\mu\in\invMeas{b}}\mu([-R,R]^c) \le \sup_{\mu\in\invMeas{b}}
\int_{|x|\ge R}\frac{e_\delta(x)}{|x|}\mu(\dif x)\le \frac{b}{R}.
\]
and the last term converges to $0$ as $R\to\infty$.
Hence $\invMeas{b}$ is tight and by Prokhorov's Theorem,
it is relatively compact.
We next show that it is also closed.
Consider a sequence $\{\mu_n\}_{n\in\N}\subset\invMeas{b}$ such that $\mu_n\to\mu$.
Set $f_m(x):=e_\delta(x) \wedge m$.  
 Since $f_m\in\cC_b(\R)$, $f_m \le e_\delta(x)$ and $\mu_n \in \invMeas{b}$, 
\[
\langle \mu, f_m\rangle = \lim_{n \to \infty}\langle \mu_n, f_m\rangle \le b, \quad \forall m>0.
\]
By monotone convergence theorem,
\[
\langle \mu, e_\delta \rangle = \lim_{m \to \infty}\langle \mu, f_m\rangle \le b.
\]
Hence, $\mu \in \invMeas{b}$ and consequently, $\invMeas{b}$ is compact.

For every $N$, $\invSetNCl{}$ is a subset of $[0,T]\times
\invMeas{N e^{K^*t }}$, hence, it is relatively compact.
Consider a sequence $\{(t_n,\mu_n)\}_{n\in\N}\subset\invSetNCl{}$ 
such that $(t_n,\mu_n)\to(t,\mu)$.
Proceeding exactly as above, we can show that
\[
\langle \mu, e_\delta \rangle = 
\lim_{n \to \infty} \lim_{m \to \infty}\langle \mu_n, f_m\rangle \le N e^{K^* t}.
\]
Hence, $(t,\mu)\in \invSetNCl{}$ and consequently, $\invSetNCl{}$ is compact.
\end{proof}

We close this section by recalling a well-known result;  
see~\cite[Theorem 30.1]{Billingsley}.
Suppose $\mu, \nu \in \cM$.  Then,
\begin{equation}
\label{eq:seperate}
\mu=\nu \quad
\Leftrightarrow \quad
\langle \mu-\nu,x^j\rangle =0, \ \ \forall \ j=1,2,\ldots
\end{equation}

\section{Viscosity solutions and test functions}
\label{s.test}

In this section,
we define viscosity sub and super-solutions
to the dynamic programming equation  \reff{eq:DPE}.
As it is standard in the viscosity theory, 
one has to first specify the class of test functions.
We continue by this selection.

\begin{dfn}
A \emph{cylindrical function} is a map of the form $(t, \mu) \mapsto F(t,\langle\mu,f\rangle)$ 
for some function $f:\R\to\R$ and  $F:[0,T]\times\R\to\R$.
This function is called \emph{cylindrical polynomial} if $f$ is a polynomial
and $F$ is continuously differentiable.
\end{dfn}

The above class is not  large
enough and we extend it to
its linear span.

\begin{dfn}
\label{def:test}
For $E \subseteq \invSetCl{}$, a \emph{viscosity test function on $E$} is a
function of the form 
\[
\varphi(t,\mu)=\sum_{j=1}^\infty \varphi_j(t,\mu),
\quad (t,\mu) \in E,
\]
where $\varphi_j$ are cylindrical polynomials
and for every $N \in \N$,
\begin{equation}
\label{eq:sum}
\lim_{M\to \infty} 
\sum_{j=M}^\infty \sup_{(t,\mu) \in E } \sum_{i=0}^{\deg{(\der {\varphi_j}})} 
\left| \langle \mu, (\der {\varphi_j})^{(i)} \rangle \right| =0.
\end{equation}
We let $\test_E$ be the set of all viscosity test functions on $E$.
\end{dfn}

 Lemma \ref{lem:continuity} below shows that
for a cylindrical polynomial $\varphi$, 
$ \langle \mu, (\der {\varphi})^{(i)}\rangle$
is uniformly bounded on $(t,\mu,a) \in \invSetNCl{} \times A$
for every $i=0,\ldots, \deg{ (\der {\varphi})}$.
Therefore, they are test functions on every $\invSetNCl{}$.

\begin{rem}
\label{r.motivation}
There are several other choices for test functions.  
In particular, we could even restrict $F$ to be 
quadratic or extended to be more general
with some integrability properties.
They all would yield equivalent definitions
and we do no pursue this equivalence here.

When $\mu_t$ is the law of
a stochastic process $X_t$ and
$\varphi$ is  a cylindrical function, 
$\varphi(t,\mu_t)
=F(t,\langle\mu_t,f\rangle)=F(t,\E[f(X_t)])$.
Then, one can employ the standard It{\^{o}} formula;
see Proposition \ref{lem:Hcontinuous} below.
\qed
\end{rem}

\begin{dfn}
 For $E \subseteq \invSetCl$ and $(t,\mu) \in E$ with $t<T$, the
 \emph{superjet} of $u$ at $(t,\mu)$ is given by,
 \[
 \supdiff{E} u(t,\mu):=\big\{(\partial_t \varphi(t,\mu), 
 \der\varphi(t,\mu,\cdot))\mid \varphi \in \test_E,  
 (u-\varphi)(t,\mu) = \max_E (u-\varphi)\big\}.
\]
The \emph{subjet} of $u$ at $(t,\mu)$ 
is defined as $\subdiff{E} u(t,\mu):=-\supdiff{E}(-u)(t,\mu)$.
\end{dfn}

\begin{dfn}
On a subspace $E \subseteq \invSetCl$,
the \emph{(sequential) upper semicontinuous envelope} of $u$ on $E$
is defined by\footnote{As $\invSet$ is first countable, semicontinuity coincides with sequential semicontinuity.}
$$
u^*_E(t,\mu):=\limsup_{E \owns (t',\mu')\to(t,\mu)}u(t,\mu),
$$
where the $\limsup$ is taken over all sequences in $E$ converging to $(t, \mu)$.
The \emph{lower semicontinuous envelope} $u_*^E$ is defined analogously.
\end{dfn}

We use the compact notations
\[
u^*:= u^*_{\invSetCl}, \ \ 
u_*:= u_*^{\invSetCl}, \quad 
u^*_N:= u^*_{\invSetNCl{}}, \ \ 
u_*^N:= u_*^{\invSetNCl{}}.
\]
We note that as opposed to the finite-dimensional cases,
when $u$ is not continuous,
the dependence of $u^*_N$ and $u_*^N$ on $N$ is non-trivial.  
This emanates from the 
fact that the interiors of all $\invSetN$ are empty.

To simplify the notation, we write $H = \sup_{a \in A} H^a$.

\begin{dfn} 
We say that a function $u:\invSetN\to\R$ is a
\emph{viscosity sub-solution} of \eqref{eq:DPE} on $\invSetN$, if for every $(t,\mu)\in\invSetN$,
\begin{align*}
-\pi_t + H(t,\mu,\pi_\mu)&\le 0\qquad \forall(\pi_t,\pi_\mu)\in  \supdiffN u^*_N(t,\mu).
\end{align*}
We say that a function $v:\invSetN\to\R$ is a
\emph{viscosity super-solution}
of \eqref{eq:DPE} on $\invSetN$, if for every $(t,\mu)\in\invSetN$,
\begin{align*}
-\pi_t + H(t,\mu,\pi_\mu)&\ge 0\qquad \forall(\pi_t,\pi_\mu)\in  \subdiffN u_*^N(t,\mu).
\end{align*}
A \emph{viscosity solution} of \eqref{eq:DPE} is a 
function on $\invSet$ that is both a sub-solution and a super-solution 
of \eqref{eq:DPE} on $\invSetN$, for every $N\in\N$. 
\end{dfn}

We continue with several  technical results.
Ultimately, we want to show some continuity properties of $H$.

\begin{dfn}
We say that $g$ has 
$\delta$-subexponential growth if 
$|g(x)x|\le \hat{C} e^{\delta|x|}$ for some $\hat{C}>0$ and every $x\in\R$.
\end{dfn}
Note that any polynomial has 
$\delta$-subexponential growth.
\begin{lem}
\label{lem:continuity}
Let $\delta>0$ be as in \ref{ass:state}.
For any $g$ with $\delta$-subexponential growth,
$$
\sup_{\mu\in\invMeas{b}}\langle\mu,|g|\rangle<\infty,
\quad
{\text{and}}
\quad
\lim_{R \to \infty} \sup_{\mu\in\invMeas{b}}\int_{|x|\ge R}|g(x)|\mu(\dif x)= 0.
$$
Moreover, there is a constant $C$, depending
only on the constants appearing in Assumption \ref{a.standing},
such that for any cylindrical polynomial $\varphi$ and $N \in \N$,
\begin{equation}
\label{eq:1}
\sup_{a \in A, (t,\mu) \in \invSetNCl{}} \left|
\langle\mu, \diffop^{a,\mu}_t[\der\varphi]\rangle \right| \le
C \sup_{(t,\mu) \in \invSetNCl{}} \sum_{i=0}^{\deg{(\der {\varphi}})} 
\left| \langle \mu, (\der {\varphi})^{(i)} \rangle \right| <\infty.
\end{equation}
\end{lem}
\begin{proof}
The estimate $\sup_{\mu\in\invMeas{b}}\langle\mu,|g|\rangle<\infty$
follows directly from Assumption \ref{a.standing} and the definition of $\invMeas{b}$.
Indeed, since $g$ has $\delta$-subexponential growth,
$$
|g(x)x|\le \hat{C} \exp(\delta |x|) \le \hat{C} e^{\delta} e_\delta(x)=: \tilde{C} e_\delta(x),
\quad x \in \R.
$$
Hence, for $R\ge 1$,
\begin{align*}
\sup_{\mu\in\invMeas{b}}\int_{|x|\ge R}|g(x)|\mu(\dif x) 
&\le \tilde{C} \sup_{\mu\in\invMeas{b}}\int_{|x|\ge R}\frac{e_\delta(x)}{|x|}\mu(\dif x)\\
&\le \frac{\tilde{C}}{R}\sup_{\mu\in\invMeas{b}}\int_{|x|\ge R}e_\delta(x)\mu(\dif x)\\
&\le \frac{b\tilde{C}}{R}.
\end{align*}

Let $f$ be a polynomial. Then,
\begin{align*}
    \langle\mu, \diffop^{a,\mu}_t[f]\rangle=
    b(t,\mu,a) \langle\mu,f^\prime \rangle &+\frac{1}{2}\sigma^2(t,\mu,a) \langle\mu,f^{\prime \prime}\rangle \\
    &+\lambda(t,\mu,a)\langle\mu,\int_{\R}(f(x+y)-f(x))\gamma(dy)\rangle.
\end{align*}
We rewrite the last term by Taylor expansion of the polynomial $f$ as follows,
$$
\langle\mu,\int_{\R}(f(x+y)-f(x))\gamma(dy)\rangle=
\sum_{i=1}^{\deg(f)}\frac{\langle\mu,f^{(i)}\rangle}{i!}\int_{\R}y^i\gamma(dy).
$$
The above equations, together with Assumption \ref{a.standing} and the fact that all derivatives of $f$ have $\delta$-subexponential growth, imply \eqref{eq:1}.
The result for a cylindrical polynomial follows similarly.
\end{proof}

\begin{prop}
\label{lem:Hcontinuous}
For every $\varphi \in \test_{\invSetN}$, $(t,\mu) \in \invSetN$ 
and $\alpha \in \cA$, 
\begin{equation}
\label{eq:kolmogrov}
\varphi(u,\mu_u)
= \varphi(t,\mu)
+ \int_{t}^u \left[\partial_t \varphi(s,\mu_s)+
\langle \mu_s, \diffop^{\alpha_s,\mu_s}_s[D_m\varphi]\rangle\right] \dif{s},
\quad u \in [t,T],
\end{equation}
where $\mu_s=\law(X^{t,\mu,\alpha}_s)$ and
 $(X^{t,\mu,\alpha}_s)_{s\in[t,T]}$ is the solution to \eqref{eq:SDE} with initial
distribution $\mu$.
Moreover, the map $(t, \mu) \mapsto H(t, \mu, \der \varphi)$ 
is continuous on $\invSetN$. 
\end{prop}
\begin{proof}
Fix $\varphi \in \test_{\invSetN}$, 
$(t,\mu) \in \invSetN$ and $\alpha \in \cA$
and let $\mu_s$ be as in the statement.
In view of Lemma \ref{lem:exp_mom}, $\mu_s \in \invSetN$
for all $s \in [t,T]$.

Let first $\varphi(\mu)=\langle\mu,f\rangle$ with $f$ polynomial, so that $D_m\varphi=f$ and $\langle \mu_s, f \rangle = \E f(X^{t,\mu,\alpha}_s)$. By stochastic calculus
$$
\langle \mu_u, f \rangle = \langle \mu, f \rangle 
+\int_{t}^u \langle \mu_s, \diffop^{\alpha_s,\mu_s}_s[f]\rangle \dif{s}.
$$
Moreover, this derivative is uniformly bounded on $\invSetN$ by the 
previous lemma.
Now consider a 
cylindrical polynomial $\varphi(t,\mu)=F(t,\langle\mu,f\rangle)$.
By calculus,
$$
\varphi(u,\mu_u) = \varphi(t,\mu) 
+\int_{t}^u \left[\partial_t \varphi(s,\mu_s) 
+F_x(s,\langle\mu_s,f\rangle) \langle \mu_s, \diffop^{\alpha_s,\mu_s}_s[f]\rangle
\right] \dif{s}.
$$
Since $D_m\varphi(s,\mu)= F_x(s,\langle\mu,s\rangle) f$, 
the above proves \eqref{eq:kolmogrov}
for cylindrical polynomials.
For a general $\varphi \in \test_{\invSetN}$, \eqref{eq:kolmogrov}
follows directly from above, the condition
\eqref{eq:sum} and the fact that $\mu_s \in \invSetN$
for all $s \in [t,T]$.

We now show continuity of $H$.
Since all derivatives of $f$ have $\delta$-subexponential growth,
Lemma \ref{lem:continuity} and the fact that $\varphi$ is a 
smooth function imply $\langle\mu, (\der\varphi)^{(i)}\rangle$ 
is continuous on every $\invSetN$, for any $i\in\N$.
In particular the uniform continuity of 
$(t, \mu) \mapsto \langle\mu, \diffop^{a,\mu}_t[\der\varphi(t,\mu)]\rangle$ 
follows from \ref{ass:bounded} and \ref{ass:convergence} and for $L$ it is assumption \ref{ass:cost}.
Hence, $H(t,\mu, \der \varphi)$ is continuous  for all cylindrical polynomials.
This continuity extends directly to all functions of the type
$\varphi^M:=\sum_{j=1}^M F_j (\langle \mu,f_j\rangle)$.

Now consider a general
test function 
$\varphi=\sum_{j=1}^\infty F_j (\langle \mu,f_j\rangle)$ and 
for $M \in \N$ set $\varphi^M:=\sum_{j=1}^M F_j (\langle \mu,f_j\rangle)$.
Since $\varphi \in \test_{\invSetN}$,
it satisfies \eqref{eq:sum}.  This together with
\eqref{eq:1} imply that
$$
\lim_{M\to \infty} \sup_{a \in A, (t,\mu) \in \invSetNCl{}  }
\sum_{j=M}^\infty  \left|\langle\mu, \diffop^{a,\mu}_t[\der\varphi_j]\rangle \right| =0.
$$
The above uniform limit enables us
to conclude that 
 $H(t,\mu, \der \varphi^M)$
converges uniformly to $H(t,\mu, \der \varphi)$
as $M$ tends to infinity.
Hence, $H(t,\mu, \der \varphi)$ is also
continuous.
\end{proof}

\section{Value function}
 \label{s.DPE_viscosity}
 
In this section we show that 
the value function $V$ is a viscosity solution to \eqref{eq:DPE}.
We start with two technical lemmata.

\begin{lem}
\label{lem:strictmax}
 For every $N\in\N$, $(t_0, \mu_0) \in \invSetN{}$, there exists a viscosity test function
 $\phi \in \test_{\invSetN{}}$
 such that $\phi(t,\mu) \geq 0$, with equality only in $(t_0, \mu_0)$,
 and
 $$
 (\phi(t_0,\mu_0),\partial_t \phi(t_0, \mu_0), \der \phi(t_0, \mu_0, \cdot)) = (0, 0, 0).
 $$
 In particular, in the definition of viscosity sub and super-solutions, 
 without loss of generality, 
we may assume that the extrema are strict.
\end{lem}

\begin{proof}
Fix $(t_0, \mu_0) \in \invSetN{}$ and set
$$
\phi(t,\mu)=\phi(t,\mu;t_0,\mu_0):=(t-t_0)^2+\sum_{j=1}^\infty \frac{1}{(j+1) 2^j}\langle \mu-\mu_0,x^j\rangle^2.
$$
By \eqref{eq:seperate}
$\phi(t,\mu)>0$ when $(t,\mu) \neq (t_0,\mu_0)$.
For any $j\in \N$, let $\varphi_j(\mu)=\frac{1}{(j+1) 2^j}\langle \mu-\mu_0,x^j\rangle^2$ and observe that
\[
 \sup_{(t,\mu) \in \invSetN{} } \sum_{i=0}^{\deg{(\der {\varphi_j}})} 
\left| \langle \mu, (\der {\varphi_j})^{(i)} \rangle \right| \le \frac{1}{2^j}K_N,
\]
for some constant $K_N$ which only depend on $\invSetN{}$.
It follows that $\phi$ satisfies \eqref{eq:sum}. 
It is clear that $\phi$ has all the claimed properties.
\end{proof}

\begin{lem}\label{lem:finite_env}
    For each $N$, $V, V^*_N$ and $V_*^N$ are bounded on $\invSetN$.
\end{lem}
\begin{proof}
Let $(t,\mu)\in\invSetN$.
From Lemma \ref{lem:exp_mom}, $\invSetN$ is invariant for \eqref{eq:SDE} and recall that $\invSetNCl{}$ is compact.
 Assumption \ref{ass:cost} and Lemma \ref{lem:continuity} implies that $|L|+|g|$ is uniformly bounded on $\invSetNCl{}$ by a constant $K_N$.
It follows that $|V(t,\mu)|\le (1+T)K_N$ on $\invSetN$.
\end{proof}

The proof of the next result is standard; \cite{BT,FS}.

\begin{thm}
\label{lem:viscosity_valuef}
Assume \eqref{eq:DPP} holds.
For any $N\in\N$, the value function
$V$ is both a viscosity sub and a super-solution to \eqref{eq:DPE} on
$\invSetN$ and
$$
V_N^*(T,\cdot)=V_*^N(T,\cdot)=G \quad
{\text{on }} \invMeas{N e^{K^*T}}.
$$
\end{thm}

\begin{proof}
Fix $N \in \N$ and note that both envelopes $V^*_N, V_*^N$
 are finite by Lemma \ref{lem:finite_env}.
 \vspace{5pt}

\step{$V^*_N$ is a viscosity sub-solution for $t<T$.}
\label{step_viscosub}
Suppose that for $\varphi \in \test_{\invSetN}$
and  $(t,\mu)\in \invSetN$,
$$
0=(V^*_N-\varphi)(t,\mu)
=\max_{\invSetN}\ (V^*_N-\varphi).
$$
Let $(t_n,\mu_n)$ be a sequence in $\invSetN$ such that 
$(t_n,\mu_n,V(t_n,\mu_n))\rightarrow (t,\mu,V^*_N(t,\mu))$.
Fix $a \in A$ and
let $(X^{t_n,\mu_n,a}_s)_{s\in[t_n,T]}$ denote the solution to  \eqref{eq:SDE} with constant control 
$ a$ and distribution $\mu_n$ at the initial time $t_n$.
For ease of notation, we set $\mu^{n,a}_s:=\law(X^{t_n,\mu_n,a}_s)$.
We use the  dynamic programming \eqref{eq:DPP}
with $\theta_n:=t_n+h$ for  $0<h< T-t$,
to obtain
$$
V(t_n,\mu_n) \le 
\int_{t_n}^{\theta_n} L(s,\mu^{n,a}_s,a) \dif{s} + V(\theta_n,\mu^{n,a}_{\theta_n})
\le  \int_{t_n}^{\theta_n} L(s,\mu^{n,a}_s,a) \dif{s} + \varphi(\theta_n,\mu^{n,a}_{\theta_n}).
$$
We pass to the limit 
to arrive at
$$
V^*_N(t,\mu)=\varphi(t,\mu) 
\le  \int_{t}^{t+h} L(s,\mu^a_s,a) \dif{s} + \varphi(t+h,\mu^a_{t+h}),
$$
where $\mu^a_s$ is the distribution
of the solution to \eqref{eq:SDE} with initial data $\mu$ at time $t$
and constant control $a$.  We now use \eqref{eq:kolmogrov}
 to obtain
$$
0\le \int_{t}^{t+h} [ 
\partial_t \varphi(s,\mu^a_s) -H^a(s,\mu^a_s,D_m\varphi)] \dif{s}.
$$
Since this holds for every $h>0$ and $a \in A$, we conclude that
$$
- \partial_t \varphi(t,\mu) +\sup_{a\in A} H^a(t,\mu,D_m\varphi) \le0.
$$
 \vspace{5pt}

\step{$V_*^N$ is a viscosity super-solution for $t<T$.} 
Suppose that there exists $(t,\mu)\in\invSetN$ and 
$\varphi\in\test_{\invSetN}$ 
such that 
$$
0=(V_*^N-\varphi)(t,\mu)
=\min_{\invSetN}\ (V^*_N-\varphi).
$$
In view of Lemma \ref{lem:strictmax},
without loss of generality we assume that
above minimum is strict.  Towards
a counterposition assume that
$$
-\partial_t\varphi(t,\mu) + H(t,\mu,\der\varphi)< 0.
$$
By the continuity of $H$ proved in Proposition \ref{lem:Hcontinuous},
there exists 
a neighbourhood $B$ of $(t,\mu)$  such that
$$
-\partial_t\varphi(t,\mu) 
-\langle \mu, \diffop^{a,\mu}_t[D_m\varphi]\rangle \le 
L(t,\mu,a),\quad \forall (t,\mu)\in B_N:=B\cap \invSetN,\ \forall a\in A.
$$
Let $(t_n,\mu_n)$ be a sequence in $\invSetN{}$ such that 
$(t_n,\mu_n,V(t_n,\mu_n))\rightarrow (t,\mu,V_*^N(t,\mu))$.
It is clear that for all large $n$, $(t_n,\mu_n)\in B_N$.
Fix an arbitrary control $\alpha\in \controlset$ and 
let $(X^{t_n,\mu_n,\alpha}_s)_{s\in[t_n,T]}$ denote the solution to  \eqref{eq:SDE} with  distribution $\mu_n$ at the initial time $t_n$.
For ease of notation, we set $\mu^{n,\alpha}_s:=\law(X^{t_n,\mu_n,\alpha}_s)$.
Consider the deterministic  times
$$
\theta_n:=\inf\{s\geq t_n:(s,\mu_s^{n,\alpha})\notin B_N \} \wedge T.
$$

By \eqref{eq:kolmogrov},
\begin{align*}
\varphi(t_n,\mu_n)&=\varphi(\theta_n,\mu_{\theta_n}^{n,\alpha})
-\int_{t_n}^{\theta_n}\left[ \partial_t\varphi(s,\mu_s^{n,\alpha})
+\langle \mu^{n,\alpha}_s, \diffop^{\alpha_s,\mu^{n,\alpha}_s}_s[D_m\varphi]\rangle\right]\dif{s}\\
&\leq\varphi(\theta_n,\mu_{\theta_n}^{n,\alpha})+\int_{t_n}^{\theta_n} L(s,\mu_s^{n,\alpha},\alpha)\dif{s}.
\end{align*}
Since $\invSetNCl{}\setminus{B_N}=\invSetNCl{}\setminus{B}$ is compact and 
$V_*^N-\varphi$ 
has a strict minimum at $(t,\mu)$, 
there exists $\eta>0$, 
independent of $\alpha$, such that 
$\varphi \le V_*^N-\eta\le V-\eta$ on $\invSetNCl{}\setminus{B}$.
Hence, the above inequality implies that
$$
\varphi(t_n,\mu_n)\leq V(\theta_n,\mu_{\theta_n}^{n,\alpha})
+\int_{t_n}^{\theta_n} L(s,\mu_s^{n,\alpha},\alpha)\dif{s}-\eta
$$
Since the $(\varphi-V)(t_n,\mu_n)\rightarrow 0$, 
for $n$ large enough,
$$
V(t_n,\mu_n)   
\leq \int_{t_n}^{\theta_n} L(s,\mu_s^{n,\alpha},\alpha)\dif{s}
+V(\theta_n,\mu_{\theta_n}^{n,\alpha})-\frac{\eta}{2}.
$$
As the above inequality holds 
with $\eta>0$ independent of $\alpha \in \controlset$,
it is in contradiction with \eqref{eq:DPP}.  Hence,
$V_*^N$ is a viscosity super-solution to \eqref{eq:DPE}.
 \vspace{5pt}

\step{$V^*_N= G$ on $\invMeas{N e^{K^*T}}.$}
Consider a sequence $\invSetNCl{} \owns (t_n,\mu_n)\to(T,\mu)$
such that $V^*_N(T,\mu)=\lim_{n\to\infty} V(t_n,\mu_n)$.
By Assumption \ref{ass:cost}, the uniform continuity of $L_1$ 
implies $\int_{t_n}^T L_1(s,\mu_s^{n,\alpha},\alpha_s)\to 0$.
Also, by Lemma \ref{lem:continuity},
the integral 
$\int_{t_n}^T L_2(\alpha_s)\langle\mu_s^{n,\alpha},L_3\rangle\le \tilde{C}(T-t_n)$ 
converges to zero.
We next show that $\mu_T^{n,\alpha}\to \mu$.
By the compactness of $\invSetNCl{}$, there exists $\hat{\mu}\in\invMeasN{}$ 
such that $\mu_T^{n,\alpha}\to \hat{\mu}$ (up to a subsequence).
It\^{o}'s Formula and Lemma \ref{lem:continuity} imply that 
$|\langle\mu_T^{n,\alpha}-\mu_n,x^j\rangle|\to 0$ for every $j\in\N$.
This implies that $\hat{\mu}=\mu$.
Hence, for an arbitrary $\alpha\in\controlset$, we have,
$$
V^*_N(T,\mu)=\lim_{n\to\infty} V(t_n,\mu_n)\le 
\lim_{n\to\infty}\bigg[ \int_{t_n}^T L(s,\mu_s^{n,\alpha},\alpha_s)+G(\mu_T^{n,\alpha})\bigg]= 
G(\mu).
$$
As $V^*_N(T,\mu) \ge V(T,\mu)=G(\mu)$, we conclude that
$V^*_N(T,\mu)=G(\mu)$.
 \vspace{5pt}

\step{$V_*^N=G$ on $\invMeas{N e^{K^*T}}$.}
Again consider $\invSetNCl{} \owns (t_n,\mu_n)\to(T,\mu)$
satisfying
$V_*^N(T,\mu)=\lim_{n\to\infty} V(t_n,\mu_n)$.
As in the previous step 
$\int_{t_n}^T L(s,\mu_s^{\alpha,n},\alpha_s)\to 0$ uniformly in $\alpha$ and 
$G(\mu_T^{n,\alpha})\to G(\mu)$, as $n\to\infty$.
For any $n\in\N$, choose $\alpha^n\in\controlset$ so that 
$V(t_n,\mu_n)\ge \int_{t_n}^T L(s,\mu_s^{n,\alpha^n},\alpha^n_s)+G(\mu_T^{n,\alpha^n})-1/n$.
This implies that
$$
V_*^N(T,\mu)=\lim_{n\to\infty} V(t_n,\mu_n)
\ge \lim_{n\to\infty} \left[\int_{t_n}^T L(s,\mu_s^{n,\alpha^n},\alpha^n_s)
+G(\mu_T^{n,\alpha^n})\right]= G(\mu).
$$
\end{proof}

\section{A comparison result} \label{s.comparison}

The following is the main comparison result.

\begin{thm}[Comparison]\label{thm:comp}
Let $u$ be an u.s.c.\ sub-solution to \eqref{eq:DPE} 
on $\invSetN$ and $v$ a l.s.c.\ super-solution to \eqref{eq:DPE} 
on $\invSetN$, satisfying $u(T,\mu)\le v(T,\mu)$ for any $(T,\mu) \in \invSetNCl{}$.
Then $u\le v$ on $\invSetNCl{}$.
\end{thm}

The following corollary is the unique characterization of the value function.
Recall, for any function $u$, we use the notation $u^*$
to denote the upper semicontinuous envelope of $u$ restricted to $\invSetCl$
and  $u_*$ is  the lower semicontinuous envelope of $u$ restricted to $\invSetCl$.

\begin{cor}
    \label{cor:uniqueness}
    The value function $V$ is the unique  viscosity solution to \eqref{eq:DPE}
    on $\invSet$ satisfying $V^*(T,\mu)
    =V_*(T,\mu)=G$ for $(T,\mu) \in \invSetCl$. Moreover, 
    $V$ restricted to $\invSetCl$ is continuous, i.e., $V^*=V_*$.
\end{cor}
\begin{proof}
We apply the above comparison result
to $V _N^*$, $V^N_*$
and use Theorem \ref{lem:viscosity_valuef} to conclude that 
the sub-solution $V^*_N$ is less than 
the super-solution $V^N_*$.  Since the
opposite inequality is immediate from their definitions,
 $V^*_N=V^N_*=:V_N$. In view of Lemma \ref{lem:envelopes}
 proved in the Appendix, this implies that
 $V^*=V_*=V$.

Let $v$ be a viscosity solution to \eqref{eq:DPE} and 
$v^*(T,\mu)=v_*(T,\mu)=G$ for $(T,\mu) \in \invSetCl$.  
Since $v_* \le v_*^N \le v^*_N \le v^*$,  we also have
$v^*_N(T,\mu)=v_*^N(T,\mu)=G$ for $(T,\mu) \in \invSetNCl{}$.
Then, the comparison result implies that
$v^*_N \le V^N_* =V_N= V_N^* \le v^N_*\le v^*_N$.
Hence, $v^*_N=v^N_*=V_N$.  This proves the 
uniqueness.
\end{proof}

The remainder of this section is devoted to the proof of Theorem~\ref{thm:comp}.
We begin by constructing a specific class of polynomials that is central to the comparison proof.
For any polynomial $f$, $\deg(f)$ is the degree of $f$.

\begin{dfn}
\label{def:invariant}
We say that a set of polynomials $\chi$ has the \emph{\textup{($*$)}-property}
if it satisfies
\begin{enumerate}
\item for any $g\in\chi$,  $g^{(i)}\in\chi$ for all $i=0,\ldots,\deg(g)$;
\item for any $g\in\chi$,  $\sum_{i=1}^{\deg(g)}m_ig^{(i)}\in\chi$ with $m_i:=\frac{1}{i!}\int_{\R}y^i\gamma(dy)$.
\end{enumerate}
Let $\Sigma$ be the  collection of all sets of polynomials that has the ($*$)-property.
\end{dfn}

Set
$$
\chi(f):= \bigcap_{\chi \in \Sigma, f \in \chi} \chi.
$$
One can directly show that 
$\chi(f)$ has the ($*$)-property
and hence it is the smallest set
of polynomials with the ($*$)-property that
includes $f$. It is also clear that for every $g \in \chi(f)$, $\chi(g) \subset \chi(f)$.

\begin{exm}
The followings are few examples of the above sets.
\begin{align*}
\chi(x)= \{&0,1,m_1,x\}\\
\chi(x^2)=\{&0,2,2m_1,2m_1^2,2x,2m_1x+2m_2,x^2\}\\
\chi(x^3)=\{&0,6,6m_1,6m_1^2,6m_1^3, 6x,6m_1x+6m_2,6m_1^2x+12m_1m_2,
3x^2,\\
& 3m_1x^2+6m_2x+6m_3, x^3\}.
\end{align*} 
\end{exm}

\begin{lem}\label{lem:finiteclass}
For any polynomial $f$, $\chi(f)$  is finite.
\end{lem} 
\begin{proof}
We show this by induction on the degree of the polynomial.
Indeed if $\deg(f)=0$, $\chi(f)=\{f,0\}$ and hence is finite.
Towards an induction proof, assume 
that we have shown that $\chi(h)$ is finite
for every polynomial $h$ with $\deg{h} \le n$
for some integer $n \ge 0$.
Let $f$ be a polynomial with $\deg(f)=n+1$.
Set $\hat{g}:=\sum_{i=1}^{\deg(f)}m_if^{(i)}$.
Then, $\deg(\hat{g})= n$ and consequently by
our assumption $\chi(\hat{g})$ is finite.
Moreover,
$$
\chi(f)=\{f\}\cup\chi(\hat{g})\cup\bigcup_{i=1}^{\deg(f)}\chi(f^{(i)}).
$$
As $\deg(f^{(i)}) \le n$ for every $i\ge 1$,
$\chi(f^{(i)})$ are finite by the induction hypothesis and therefore,
$\chi(f)$ is also finite.
\end{proof}

Set $\Theta:=\cup_{j=1}^{\infty}\chi(x^j)$.
Then, $\Theta$ contains all monomials $\{x^j\}_{j=1}^{\infty}$,
it is countable and
$\chi(f) \subset \Theta$ for every $f \in \Theta$.  
Let $\{f_j\}_{j=1}^{\infty}$ be an enumeration of $\Theta$.

The definition of $\invMeas{b}$ and Lemma \ref{lem:continuity} imply that
$$
s_j(b):=1+\sup_{\mu\in\invMeas{b}}\langle\mu,f_j\rangle^2 <\infty,\quad
\forall j=1,2,\ldots.
$$
As $\chi(f) \subset \Theta$ for every $f \in \Theta$, 
there exists a finite index set $I_j$ satisfying,
$$
\chi(f_j)=\left\{f_i \mid i\in I_j\right\}\quad
j=1,2,\ldots
$$
Fix $b>0$ and set
\begin{equation}
\label{def:coeff}
c_j(b):=\big(\sum_{k\in I_j} 2^k\big)^{-1} \big(\sum_{k\in I_j} s_k(b)\big)^{-2}.
\end{equation}
Since  $f_j \in \chi(f_j)$, $j\in I_j$ and  therefore, $c_j(b)\le 2^{-j}$. 
Hence, $\sum_{j=1}^{\infty} c_j(b) \le 1$.
Also, for each  $i \in I_j$, $f_i \in \chi(f_j)$ and consequently,
$\chi(f_i)\subset \chi(f_j)$.  This implies that $I_i\subset I_j$.
Moreover, $s_j(b) \ge 1$.  Hence,
the definition  \eqref{def:coeff}  implies that
\begin{equation}\label{eq: coeff}
c_j(b)\le c_i(b),\qquad \forall i\in I_j.
\end{equation}
Finally, observe that, by the definitions of $s_j(b)$ and $c_j(b)$,
\begin{equation}\label{eq:uniform_serie}
\sum_{j=1}^{\infty}c_j(b)\langle\mu,f_j\rangle^2\le 1, \quad
\forall \mu \in \invMeas{b}.
\end{equation}

\begin{proof}[Proof of Theorem \ref{thm:comp}]

Fix $N \in \N$.

To simplify the notation
we write $c_j$ for $c_j(Ne^{K^*T})$.  In particular,
for any $(t,\mu) \in \invSetNCl{}$, $\mu \in \invMeas{ Ne^{K^*t}} \subset  \invMeas{ Ne^{K^*T}}$
and therefore, by \eqref{eq: coeff}
$$
\sup_{(t,\mu) \in \invSetNCl{}}\sum_{j=1}^{\infty}c_j\langle\mu,f_j\rangle^2\le 1.
$$

Towards a counterposition, 
suppose that $\sup_{{\invSetNCl{}}}(u-v)>0$.
Since $u-v$ is u.s.c. and $\invSetNCl{}$ is weak$^*$ compact, the maximum
$$
\ell:= \max_{(t,\mu)\in\invSetNCl{}} \big( (u-v)(t,\mu) - 2 \eta (T-t) \big)
$$
is achieved and $\ell>0$ for all sufficiently small $\eta \in(0,\eta_0]$.

\step{Doubling of variables.}
\label{step:doubling}
Recall  $\Theta=\{f_j\}_{j=1}^\infty$ and the constants $\{c_j\}$  in \eqref{def:coeff}
with $b=Ne^{K^*T}$.
For $n\in\N$, $\varepsilon>0$,  $\eta \in( 0,\eta_0]$ set
\begin{align*}
\phi_{\varepsilon}(t,\mu,s,\nu)&:=u(t,\mu)-v(s,\nu) -\frac{1}{\varepsilon}\ds{\mu-\nu} - \beta_{\eta, \varepsilon}(t, s),\\
\beta_{\eta, \varepsilon}(t, s) & :=   \eta (T - t + T - s) + \frac{1}{\varepsilon}(t-s)^2.
\end{align*}
By our assumptions, $\phi_{\varepsilon}$ 
admits a maximizer $(t^*_{\varepsilon},\mu^*_{\varepsilon}, s^*_{\varepsilon}, \nu^*_{\varepsilon})$
satisfying,
\begin{equation}
\label{eq:lower_bound}
\phi_{\varepsilon}(t^*_{\varepsilon},\mu^*_{\varepsilon}, s^*_{\varepsilon}, \nu^*_{\varepsilon})
=\max_{{\invSetNCl{}}} \phi_{\varepsilon} \ge\ell >0.
\end{equation}

Since $\invSetNCl{}$ is compact and 
$u$ is u.s.c., $M:= \max_{{\invSetNCl{}}} u \in \R$.
As $v$ is l.s.c., similarly
 $m:= \min_{{\invSetNCl{}}} v \in \R$.
In view of \eqref{eq:lower_bound},
$$
 0\le \frac{\zeta_\epsilon}{\varepsilon}:= \frac{1}{\varepsilon}\big[ \ds{\mu^*_{\varepsilon}-\nu^*_{\varepsilon}} 
 + (t^*_{\varepsilon} - s^*_{ \varepsilon})^2 \big]\le M-m-\ell=:C<\infty.
$$

As $\invSetNCl{}$ is compact, there exist subsequences 
$\{(t^*_{\varepsilon_i},\mu^*_{\varepsilon_i}),(s^*_{\varepsilon_i},\nu^*_{\varepsilon_i})\}_{i\in\N}$ 
such that $\mu^*_{\varepsilon_i}$ and $\nu^*_{\varepsilon_i}$ 
converge to $\mu^*$ and $\nu^*$ respectively, and 
$t^*_{\varepsilon_i}$ and $s^*_{\varepsilon_i}$ both converge to $t^*$.

\step{$\nu^*=\mu^*$.}  Since $\zeta_\epsilon$ converges to zero, 
$\langle \mu^*_{\varepsilon}-\nu^*_{\varepsilon}, f_j\rangle$ converges to zero for each $j$.  
As $\Theta=\{f_j\}_{j=1}^\infty$ contains all the monomials,
$\lim_{\varepsilon\to 0}\langle\mu^*_{\varepsilon}-\nu^*_{\varepsilon},x^j\rangle=0$ for any $j\in\N$.
In view of Lemma \ref{lem:continuity}, the map
$\mu \mapsto \langle \mu,x^j\rangle$ is continuous on $\invSetNCl{}$.  Hence,
$$
\langle \mu^*-\nu^*, x^j\rangle =
 \lim_i \langle \mu^*_{\varepsilon_i}-\nu^*_{\varepsilon_i}, x^j\rangle =0, \quad
 j=1,2,\ldots
 $$
By \eqref{eq:seperate}, we conclude that $\nu^*=\mu^*$.

\step{$t^* < T$.}
Towards a counterposition, assume that
$t^*=T$.  Since by hypothesis $(u-v)(T, \cdot)  \leq 0$, $v$ is l.s.c., and $u$ is u.s.c.,
\begin{align*}
0\ge (u-v)(T,\mu^*) &\ge \limsup_i
u(t_{\varepsilon_i},\mu^*_{\varepsilon_i})-
v(s_{\varepsilon_i},\nu^*_{\varepsilon_i})\\
& \ge  \limsup_i \phi_{\varepsilon_i}(t^*_{\varepsilon_i},\mu^*_{\varepsilon_i}, s^*_{\varepsilon_i}, \nu^*_{\varepsilon_i})
\ge \ell >0.
\end{align*}

\step{We claim that $\limsup_{i\rightarrow\infty}\frac{\zeta_{\varepsilon_i}}{\varepsilon_i}=0$.}
Indeed,
\begin{align*}
    \ell
    &\ge \phi_{\epsilon}(t^*,\mu^*,t^*,\mu^*)\\
    &=u(t^*,\mu^*)-v(t^*,\mu^*) - 2 \eta (T-t^*) \\
    &\ge \limsup_{i \to \infty}\big( u(t^*_{\varepsilon_i},\mu^*_{\varepsilon_i})-v(s^*_{\varepsilon_i},\nu^*_{\varepsilon_i}) - \eta (T-t^*_{\varepsilon_i} + T-s^*_{\varepsilon_i})\big)\\
    &\ge\ell +\limsup_{i \to \infty} \frac{1}{\varepsilon_i} \bigg( \ds{\mu_{\varepsilon_i}^*-\nu_{\varepsilon_i}^*} + (t^*_{\varepsilon_i} - s^*_{\varepsilon_i})^2 \bigg)\\
    &=\ell +\limsup_{i \to \infty}\frac{\zeta_{\varepsilon_i}}{\varepsilon_i}.
\end{align*}
Hence we conclude that
\begin{equation}\label{eq.lambdan2}
\limsup_{i\rightarrow\infty}\frac{\zeta_{\varepsilon_i}}{\varepsilon_i}=0.
\end{equation}

\step{Initial Estimate.}
\label{step:initial}
Let $\{\mu^*_{\varepsilon}\}$, $\{\nu^*_{\varepsilon}\}$ 
as in Step \ref{step:doubling} and set
$$
\pi^*_{\varepsilon}(\cdot)
:= \frac{2}{\varepsilon}\dsder{\mu_{\varepsilon}^*-\nu_{\varepsilon}^*}(\cdot). 
$$
Note that 
$$
\pi^*_{\varepsilon}(\cdot)=\der \varphi_1(\mu_{\varepsilon}^*,\cdot)=-\der \varphi_2(\nu_{\varepsilon}^*,\cdot),
$$
where $\varphi_1(\mu):= \frac{1}{\varepsilon} \ds{\mu-\nu_{\varepsilon}^*}$, respectively, $\varphi_2(\mu):= \frac{1}{\varepsilon} \ds{\mu_{\varepsilon}^*-\mu}$.
One can directly verify that $\varphi_1$ and $\varphi_2$ are test functions on $\invSetNCl{}$, i.e., $\varphi_1, \varphi_2 \in \test_{\invSetNCl{}}$.
We thus have,
$$
(\partial_t \beta_{\eta, \varepsilon}(t^*_{\varepsilon}, s^*_{\varepsilon}), \pi^*_{\varepsilon})
\in \jet{1}{+}u(t^*_{\varepsilon}, \mu^*_{\varepsilon}), \quad
(-\partial_s \beta_{\eta, \varepsilon}(t^*_{\varepsilon}, s^*_{\varepsilon}), \pi^*_{\varepsilon}) 
\in \jet{1}{-}v(s^*_{\varepsilon}, \nu^*_{\varepsilon}).
$$
Then, by the viscosity properties of $u$ and $v$,
$$
-\partial_t \beta_{\eta, \varepsilon}(t^*_{\varepsilon}, s^*_{\varepsilon})   
+ H(t^*_{\varepsilon},\mu^*_{\varepsilon},\pi^*_{\varepsilon}) \leq 0, \quad
 \partial_s \beta_{\eta, \varepsilon}(t^*_{\varepsilon}, s^*_{\varepsilon}) 
 + H(s^*_{\varepsilon},\nu^*_{\varepsilon},\pi^*_{\varepsilon})\geq0.
 $$
We combine and use the definition of $\beta_{\eta, \varepsilon}$ to arrive at
\begin{align*}
0 < 2 \eta &\le  
H(s^*_{\varepsilon},\nu^*_{\varepsilon}, \pi^*_{\varepsilon})
- H(t^*_{\varepsilon},\mu^*_{\varepsilon}, \pi^*_{\varepsilon})\\
& =  \sup_{a \in A} H^a(s^*_{\varepsilon},\nu^*_{\varepsilon}, \pi^*_{\varepsilon})
-  \sup_{a \in A} H^a(t^*_{\varepsilon},\mu^*_{\varepsilon}, \pi^*_{\varepsilon})\\
& \le \sup_{a \in A}(H^a(s^*_{\varepsilon},\nu^*_{\varepsilon}, \pi^*_{\varepsilon})
-H^a(t^*_{\varepsilon},\mu^*_{\varepsilon}, \pi^*_{\varepsilon}))
=:\sup_{a \in A}I^a.
\end{align*}
Moreover,
\begin{align*}
    I^a    &:=
    L(t^*_{\varepsilon},\mu^*_{\varepsilon},a) - L(s^*_{\varepsilon}, \nu^*_{\varepsilon}, a) 
+\langle\mu^*_{\varepsilon}, \diffop^{a,\mu^*_{\varepsilon}}_{t^*_{\varepsilon}}[\pi^*_{\varepsilon}]\rangle
     - \langle\nu^*_{\varepsilon}, \diffop^{a,\nu^*_{\varepsilon}}_{s^*_{\varepsilon}}[\pi^*_{\varepsilon}]\rangle \\
     &=
     \vphantom{\underbrace{L}_{}}\smash{\underbrace{L(t^*_{\varepsilon},\mu^*_{\varepsilon},a) - L(s^*_{\varepsilon}, \nu^*_{\varepsilon}, a)}_{I_1^a}
    + \underbrace{\langle\mu^*_{\varepsilon} - \nu^*_{\varepsilon},
    \diffop^{a,\mu^*_{\varepsilon}}_{t^*_{\varepsilon}}[\pi^*_{\varepsilon}]\rangle}_{I_2^a}
    + \underbrace{\langle\nu^*_{\varepsilon}, 
    \diffop^{a,\mu^*_{\varepsilon}}_{t^*_{\varepsilon}}[\pi^*_{\varepsilon}] 
- \diffop^{a,\nu^*_{\varepsilon}}_{s^*_{\varepsilon}}[\pi^*_{\varepsilon}]\rangle}_{I_3^a}}.
\end{align*}
By Assumption \ref{ass:cost} and Lemma \ref{lem:continuity},
 $\lim_{\varepsilon \to 0} \sup_{a \in A}I^a_1 \to 0$.
 
 \step{Estimate of $I_2$.}
We rewrite the second term as,
$$
I_2 := \sup_{a \in A} I_2^a \leq \sup_{a \in A} \frac{2}{\varepsilon} 
 \sum_{j=1}^\infty c_j |\langle \mu_{\varepsilon}^*-\nu_{\varepsilon}^*, f_j\rangle 
 \langle \mu_{\varepsilon}^*-\nu_{\varepsilon}^*, 
 \diffop^{a, \mu^*_{\varepsilon}}_{t^*_{\varepsilon}}[ f_j]\rangle |
 \leq I_2^b+I_2^\sigma +I_2^\gamma,
 $$
related to the three terms
appearing in the generator $\diffop^{a, \mu^*_{\varepsilon}}_{t^*_{\varepsilon}}$, which appear explicitly below.

By construction, for every $j\in\N$, 
there exists an index $k_1(j)$ such that $f_j^\prime = f_{k_1(j)}$.
Also, as $f_j^\prime=f_{k_1(j)} \in \chi(f_j)$, $\chi(f_{k_1(j)}) \subset \chi(f_j)$,
and consequently, $I_{k_1(j)} \subset I_j$.
Therefore, the definition
\eqref{def:coeff} yields that $c_j \leq c_{k_1(j)}$.
We now directly estimate using these and \ref{ass:bounded} to obtain,
\begin{align*}
    I_2^b
    &= \sup_{a \in A} \frac{2}{\varepsilon} 
    \sum_{j=1}^\infty c_j |\langle \mu_{\varepsilon}^*-\nu_{\varepsilon}^*, f_j\rangle 
    \langle \mu_{\varepsilon}^*-\nu_{\varepsilon}^*, b(t^*_{\varepsilon}, \mu^*_{\varepsilon},a) f_j^\prime\rangle| \\
    &\leq C \frac{2}{\varepsilon} \sum_{j=1}^\infty c_j |\langle \mu_{\varepsilon}^*-\nu_{\varepsilon}^*, f_j\rangle 
    \langle \mu_{\varepsilon}^*-\nu_{\varepsilon}^*, f_j^\prime\rangle | \\
    &\leq C \frac{2}{\varepsilon}
    \bigg( \sum_{j=1}^\infty c_j \langle \mu_{\varepsilon}^*-\nu_{\varepsilon}^*, f_j\rangle^2
     + \sum_{j=1}^\infty c_{k_1(j)} \langle \mu_{\varepsilon}^*-\nu_{\varepsilon}^*, f_{k_1(j)} \rangle^2 \bigg)\\
     &\leq C \frac{4}{\varepsilon}
     \sum_{j=1}^\infty c_j \langle \mu_{\varepsilon}^*-\nu_{\varepsilon}^*, f_j\rangle^2,
\end{align*}
which converges to 0, by \eqref{eq.lambdan2}.

We estimate $I^\sigma_2$ similarly.
Indeed, for every $j\in\N$, there exists an index $k_2(j)$ such that 
$f_j^{\prime \prime} = f_{k_2(j)}$ and $c_j \leq c_{k_2(j)}$.
Then, 
\begin{align*}
    I_2^\sigma
    &= \sup_{a \in A} \frac{2}{\varepsilon} 
    \sum_{j=1}^\infty c_j |\langle \mu_{\varepsilon}^*-\nu_{\varepsilon}^*, f_j\rangle 
    \langle \mu_{\varepsilon}^*-\nu_{\varepsilon}^*, 
    \sigma(t^*_{\varepsilon}, \mu^*_{\varepsilon},a) f_j^{\prime \prime} \rangle| \\
    &\leq C \frac{2}{\varepsilon} 
    \sum_{j=1}^\infty c_j |\langle \mu_{\varepsilon}^*-\nu_{\varepsilon}^*, f_j\rangle 
    \langle \mu_{\varepsilon}^*-\nu_{\varepsilon}^*, f_j^{\prime \prime} \rangle | \\
    &\leq C \frac{2}{\varepsilon}
    \bigg( \ds{\mu_{\varepsilon}^*-\nu_{\varepsilon}^*}
     + \sum_{j=1}^\infty c_{k_2(j)} \langle \mu_{\varepsilon}^*-\nu_{\varepsilon}^*, f_{k_2(j)} \rangle^2 \bigg),
\end{align*}
which also converges to 0, by \eqref{eq.lambdan2}.

We analyse $I_2^\gamma$ next. By the Taylor expansion of $f_j$,
\begin{align*}
g_j(x) &:=\int_{\R}[f_j(x+y)-f_j(x)]\gamma(dy)\\
 &=\sum_{i=1}^{\deg(f_j)}\frac{f_j^{(i)}(x)}{i!}\int_{\R}y^i\gamma(dy) 
 =\sum_{i=1}^{\deg(f_j)}m_if_j^{(i)}(x).
\end{align*}
Again, by the construction of $\{f_j\}$, for all $j\in\N$, 
there exists $k_\lambda(j)$ such that $g_j = f_{k_\lambda(j)}$ and $c_j \leq c_{k_\lambda(j)}$.
Hence,
\begin{align*}
    I_2^\lambda
    &= \sup_{a \in A}
    \frac{2}{\varepsilon} \sum_{j=1}^\infty c_j |\langle \mu_{\varepsilon}^*-\nu_{\varepsilon}^*, f_j\rangle \langle \mu_{\varepsilon}^*-\nu_{\varepsilon}^*, \lambda(t^*_{\varepsilon}, \mu^*_{\varepsilon},a) g_j\rangle| \\
    &\leq C \frac{2}{\varepsilon} \sum_{j=1}^\infty c_j |\langle \mu_{\varepsilon}^*-\nu_{\varepsilon}^*, f_j\rangle \langle \mu_{\varepsilon}^*-\nu_{\varepsilon}^*, g_j\rangle | \\
    &\leq C \frac{2}{\varepsilon}
    \bigg( \ds{\mu_{\varepsilon}^*-\nu_{\varepsilon}^*}
     + \sum_{j=1}^\infty c_{k_\lambda(j)} \langle \mu_{\varepsilon}^*-\nu_{\varepsilon}^*, f_{k_\lambda(j)} \rangle^2 \bigg).
\end{align*}
As this quantity also vanishes as $\varepsilon \to 0$, 
we conclude that $I_2 \to 0$ as $\varepsilon$ goes to zero.

\step{Estimating $I_3$.}
As in the previous step, we write  
$$
I_3= \sup_{a \in A} \langle\nu^*_{\varepsilon}, 
    \diffop^{a,\mu^*_{\varepsilon}}_{t^*_{\varepsilon}}[\pi^*_{\varepsilon}] 
    - \diffop^{a,\nu^*_{\varepsilon}}_{s^*_{\varepsilon}}[\pi^*_{\varepsilon}]\rangle
    \leq I_3^b+I_3^\sigma +I_3^\gamma
 $$  
related to the three terms appearing in the generator.
Since the estimates of each term is very similar to each other,
we provide the details of only the first one.

By \ref{ass:convergence}, there exists $C_1$ such that 
$$
\big(b(t^*_{\varepsilon},\mu^*_{\varepsilon},a)
-b(s^*_{\varepsilon},\nu^*_{\varepsilon},a))^2
\le C_1(t^*_{\varepsilon}-s^*_{\varepsilon})^2
+C_1\ds{\mu^*_{\varepsilon}-\nu^*_{\varepsilon}}.
$$
It follows,
\begin{align*}
| I^b_3| &\le \sup_{a \in A} \frac{2}{\varepsilon}
\sum_{j=1}^\infty c_j
 \left| \langle \mu_{\varepsilon}^*-\nu_{\varepsilon}^*, f_j\rangle \
 \langle \nu_\varepsilon^*, (b(t^*_{\varepsilon},\mu^*_{\varepsilon},a)
-b(s^*_{\varepsilon},\nu^*_{\varepsilon},a)) f_j^\prime\rangle \right|\\
 &\le  \frac{2}{\varepsilon}\ds{\mu^*_{\varepsilon}-\nu^*_{\varepsilon}}\\
 &\hspace{10pt}
 + \frac{2C_1}{\varepsilon}
 \bigg((t^*_{\varepsilon}-s^*_{\varepsilon})^2
 +\ds{\mu^*_{\varepsilon}-\nu^*_{\varepsilon}}\bigg) \ds[f'_j]{\nu^*_{\varepsilon}} .
\end{align*}
Note that  by  \eqref{eq:uniform_serie},
$$
 \ds[f'_j]{\nu^*_{\varepsilon}} \le 
  \ds[f_j]{\nu^*_{\varepsilon}} \le 1.
  $$
  Hence,
 $$ 
 |I^b_3|  \le 
 \frac{4C_1}{\varepsilon}\bigg((t^*_{\varepsilon}-s^*_{\varepsilon})^2
+\ds{\mu^*_{\varepsilon}-\nu^*_{\varepsilon}}\bigg).
$$
In view of  \eqref{eq.lambdan2}, 
we conclude that $I^b_3$ goes to zero as $\varepsilon \to 0$.
Repeating the same argument for $I^\sigma_3$ and $I^\lambda_3$, we conclude that $I_3$
also converges to zero.

\step{Conclusion}  In Step \ref{step:initial} we have shown that
$$
0<2 \eta \le \sup_{a \in A} I^a =I_1+I_2+I_3.
$$  
In the preceding steps
we have shown that each of the three terms converge to zero as 
$\varepsilon$ tends to zero. 
Clearly this contradicts with the fact that $ \eta>0$.
\end{proof}

\appendix

\section{Solutions of controlled McKean--Vlasov SDEs}\label{s.solutionSDE}
For completeness,  we provide here an existence result for the McKean--Vlasov SDE \eqref{eq:SDE}.

Using the functions and coefficients of Section~\ref{s.comparison},
we  fix $b>0$ and  start by proving
functional analytic properties of $\invMeas{b}$.
Set
\[
d(\mu,\nu;b):=\sum_{j=1}^\infty c_j(b) |\langle \mu-\nu,f_j\rangle|,
\quad \mu,\nu\in\invMeas{b}.
\]

\begin{lem}\label{lem:metric}
A sequence $\{\mu_n\}_{n\in\N}$ in  $\invMeas{b}$ converges weakly
 to $\mu \in \invMeas{b}$ if and only if $\lim_{n \to \infty}\dist(\mu_n,\mu;b)= 0$. 
\end{lem}
\begin{proof}
As $\Theta$ contains all monomials,
in view of \eqref{eq:seperate}, 
$\dist(\mu,\nu;b)=0$ if and only if $\mu=\nu$,
and one can then directly
verify that $\dist$ is a metric on $\invMeas{b}$. 
Moreover, since $\sum_{j=1}^\infty c_j(b) \le 1$, 
by \eqref{eq:uniform_serie},  $\dist \le 1$ on $\invMeas{b}$. 
Suppose $\mu_n\to\mu$ as $n\to\infty$. By dominated convergence,
\[
\lim_{n\to\infty} \dist(\mu_n,\mu;b)=\sum_{j=1}^\infty c_j(b) \lim_{n\to\infty} |\langle \mu_n-\mu,f_j\rangle|=0,
\]
where the last equality follows from Lemma \ref{lem:continuity}.
Now suppose $\dist(\mu_n,\mu;b)\to 0$ as $n\to\infty$. 
Since $\invMeas{b}$ is compact, the sequence $\{\mu_n\}$
has limit points and since $d$ is a metric, we conclude that
it can only have one limit point $\mu$.
\end{proof}

We next fix $t \in [0,T]$ and consider the space
\[
\cX_t(b):=\big\{\bar{\mu}=(\mu_s)_{s\in[t,T]}\mid \mu_s\in\invMeas{b},\ \ \forall s\in[t,T]
\big\},
\]
and the function
\[
\dsup{T}(\mu,\nu;b)=\sup_{t\le s\le T}\dist(\mu_s,\nu_s;b).
\]
It is straightforward to see that $\dsup{T}$ is a metric on $\cX_t(b)$.
\begin{lem}\label{lem:complete}
$(\cX_t(b),\dsup{T})$ is a complete metric space.
\end{lem}
\begin{proof}
Let $\{\bar{\mu}^n\}_{n\in\N}$ be a Cauchy sequence. In particular $\{\mu_s^n\}_{n\in\N}$ 
is a Cauchy sequence in $(\invMeas{b},\dist)$ for any $s\in[t,T]$ 
and by Lemma \ref{lem:metric}, there exists $\mu_s\in\invMeas{b}$ 
such that $\mu^n_s\to\mu_s$ as $n\to\infty$. 
We claim that $\bar{\mu}:=(\mu_s)_{s\in[t,T]}$ is the limit of $\{\bar{\mu}^n\}_{n\in\N}$.
Indeed, for $\varepsilon >0$, there is $\bar{n}$ such that
$\dist(\mu^n_s,\mu^m_s;b)\le\varepsilon$ 
for every $n,m\ge\bar{n}$ and $s\in[t,T]$. 
By letting $m$ tend to infinity and by using the previous lemma, 
we conclude that $\dist(\mu^n_s,\mu_s;b)\le\varepsilon$ for any $s\in [t,T]$. 
The result follows after  taking the supremum over $s\in [t,T]$. 
\end{proof}

This structure allows us to study the McKean--Vlasov equation \eqref{eq:SDE}.
For similar results we refer to the book of Carmona \& Delarue \cite{CD}
and the references therein.

\begin{thm} Under Assumption \ref{a.standing}, for any 
$(t,\mu) \in \invSet$ and
control $\alpha\in \controlset$, the equation \eqref{eq:SDE} 
with initial data $X_t \sim \mu$
has a unique solution.
\end{thm}
\begin{proof}
Fix $(t,\mu) \in \invSet$
and control $\alpha \in \controlset$.
There is $N\in\N$ such that $\mu \in \invMeas{Ne^{K^*t}}$. 
Let $\cX:= \cX_t(Ne^{K^*(T-t)})$
and $c_j:= c_j((Ne^{K^*(T-t)})$.

For any $\bar{\mu}=(\mu_s)_{s\in[t,T]}\in\cX$, set 
\[
X^{\bar\mu}_s
:=\int_t^s b(r,\mu_r,\alpha_r)dr+\int_t^s \sigma(r,\mu_r,\alpha_r) \dif W_r
+\sum_{t\le r\le s}\Delta J_r, 
\]
with distribution $\mu$ at time $t$,
and
$$
\Phi:\cX\to\cX,\quad \bar{\mu}\mapsto 
\Phi(\bar{\mu}):=(\law(X^{\bar\mu}_s))_{s\in[t,T]}.
$$
Recall that the set $\invSetNCl{}$ in Lemma $\ref{lem:exp_mom}$ is invariant for \eqref{eq:SDE}.
Therefore,  $\Phi(\bar{\mu}) \in \cX$.
Moreover, the law of any solution to \eqref{eq:SDE} is a fixed point of $\Phi$.

To simplify the notation for $\bar{\mu},\bar{\mu}'\in\cX$, 
let $\bar{\nu}=\Phi(\bar{\mu}),\bar{\nu}'=\Phi(\bar{\mu}')$.
Consider now $f_j\in\Theta$.  We now apply  It\^o's Formula 
to arrive at,

\begin{align*}
f_j(X^{\bar\mu}_s)&= X^{\bar\mu}_t+\int_t^s b(r,\mu_r,\alpha_r) f_j'(X^\alpha_r)dr\\
&+\frac{1}{2}\int_t^s\sigma^2(r,\mu_r,\alpha_r)f_j''(X^\alpha_r) dr\\
  &+\int_t^s\sigma(r,\mu_r,\alpha_r)f_j'(X^\alpha_r) \dif W_r+\sum_{t\le r\le s}f_j(\Delta J_r).
\end{align*}
From Assumption \ref{ass:bounded}, the stochastic integral in the above formula is a local martingale.
Denote by $\{\tau_n\}_{n\in\N}$ a localizing sequence and take expectation on both sides. Recalling that $\alpha$ is deterministic, we obtain:
\begin{align*}
\E[f_j(X^{\bar\mu}_{s\wedge\tau_n})]&= X^{\bar\mu}_t+\int_t^s b(r,\mu_r,\alpha_r) \E[f_j'(X^\alpha_r)1_{t\le r\le \tau_n}]dr\\
&+\frac{1}{2}\int_t^s\sigma^2(r,\mu_r,\alpha_r)\E[f_j''(X^\alpha_r)1_{t\le r\le \tau_n}] dr+\E\big[\sum_{t\le r\le s\wedge\tau_n}f_j(\Delta J_r)\big].
\end{align*}
By dominated convergence, the equality pass to the limit as $n\to\infty$.
For ease of notation, denote $\Delta b(r):=b(r,\mu_r,\alpha_r)-b(r,\mu'_r,\alpha_r)$ and similarly $\Delta \sigma^2(r)$ and $\Delta \lambda(r)$.
From  $\langle \nu_s,f_j\rangle=\E[f_j(X^{\bar\mu}_s)]$, we deduce
\begin{alignat*}{2}
     \langle \nu_s-\nu'_s,f_j\rangle&=&& \int_t^s \Delta b(r) \langle\nu,f_j'\rangle dr+\int_t^s b(r,\mu'_r,\alpha_r) \langle\nu-\nu',f_j'\rangle dr\\
     &&&+\frac{1}{2}\int_t^s \Delta \sigma^2(r) \langle\nu,f_j''\rangle dr+\frac{1}{2}\int_t^s \sigma^2(r,\mu'_r,\alpha_r) \langle\nu-\nu',f_j''\rangle dr\\
     &&&+ \int_t^s \Delta \lambda(r) \langle\nu,g_j\rangle dr+\int_t^s \lambda(r,\mu'_r,\alpha_r) \langle\nu-\nu',g_j\rangle dr,
\end{alignat*}
where $g_j=\sum_{i=1}^{\deg(f_j)}m_if_j^{(i)}$ with $m_i:=\frac{1}{i!}\int_{\R}y^i\gamma(dy)$.
Recall now that the collection of coefficients $\{c_j\}_{j\in\N}$ satisfies \eqref{eq: coeff}, so that,
\[c_j\le k_j:=\min\{c_{j_1},c_{j_2},c_{j_g}\},
\]
where $c_{j_1}$, $c_{j_2}$ and $c_{j_g}$ are the coefficients of $f'_j$, $f''_j$ and $g_j$ respectively.
We can therefore multiply by $k_j$ both sides of the above equality to get, using also Assumption \ref{ass:bounded} and \ref{ass:convergence},
\begin{alignat*}{2}
     c_j|\langle \nu_s-\nu'_s,f_j\rangle|&\le{}&&  k_j|\langle \nu_s-\nu'_s,f_j\rangle|\\
     &\le{}&&  \bar{C} \int_t^s \dist(\mu_r,\mu'_r) \big(c_{j_1}|\langle\nu,f_j'\rangle|+c_{j_2}|\langle\nu,f_j''\rangle| +c_{g_j}|\langle\nu,g_j\rangle|\big)dr\\
     &&&+\int_t^s c_{j_1}|\langle\nu-\nu',f_j'\rangle| + c_{j_2}|\langle\nu-\nu',f_j''\rangle| dr+\int_t^s c_{g_j}|\langle\nu-\nu',g_j\rangle| dr,
\end{alignat*}
for some constant $\bar{C}$ which depends only on the coefficients of \eqref{eq:SDE}.
By summing up over $j\in\N$ and recalling \eqref{eq:uniform_serie}, we obtain
\[
\dist(\nu_s,\nu'_s)\le 3\bar{C}\bigg(\int_t^s\dist(\mu_r,\mu'_r)+\int_t^s\dist(\nu_r,\nu'_r)\bigg).
\]
Using Gronwall's Lemma, we obtain
\[
\dsup{s}(\Phi(\bar{\mu}),\Phi(\bar{\mu}'))\le e^{3\bar{C}s}\int_t^s\dsup{r}(\bar{\mu},\bar{\mu}'),
\]
for any $t\le s\le T$.
Denoting now $C(s):=e^{3\bar{C}s}$ and $\Phi^k$ the composition of $k$ times the map $\Phi$, it can be verified, by induction,
\[
\dsup{T}(\Phi^k(\bar{\mu}),\Phi^k(\bar{\mu}'))\le \frac{C(T)^kT^k}{k!}\dsup{T}(\bar{\mu},\bar{\mu}').
\]
For $k$ large enough $\Phi^k$ is a contraction on $\cX$, which is a complete metric space 
in view of Lemma \ref{lem:complete}.
Thus, the map $\Phi$ admits a unique fixed point.
\end{proof}

\section{Semicontinuous envelopes}
\label{sec:envelope}

In this section, we show that
the semicontinuous envelopes
defined on $\invSetN$ converge to the envelopes
defined on $\invSet$.

\begin{lem}\label{lem:envelopes}
    Let $(E, \tau)$ be a topological space and $(E_N, \tau_N)_{N \in \N}$ 
    a sequence of topological spaces with $(E_N)_{N \in \N}$ increasing to $E$, i.e.,
    $\cup_{n\in\N} E_N = E$ and $E_N\subset E_{N+1}$ for any $N$. Let $\tau_N$ the subspace topology induced by $\tau$.
    Denote by $u^* : E \to \R\cup\{\infty\}$ the upper semicontinuous envelope 
    on $(E, \tau)$ and by $u_N^* : E_N \to \R\cup\{\infty\}$ the upper semicontinuous envelope on $(E_N, \tau_N)$.
    Then, $\lim_{N \to \infty} u_N^* = u^*$.
    Similarly, if $u^N_*$ is the lower semicontinuous envelope on $(E_N, \tau_N)$, then $\lim_{N \to \infty} u^N_* = u_*$
\end{lem}
\begin{proof}
    Consider the following representations of the semicontinuous envelopes.
    Let $U(\mu)$ be the collection of $\tau$-neighborhoods of $\mu$.
    Then, since $E_N$ is endowed with the subspace topology,
    for any $N \in \N$,
    \begin{alignat*}{3}
        u^*(\mu) &= \inf_{W \in U(\mu)} \sup_{W} u, && \quad \text{for }\mu \in E, \\
        u_N^*(\mu) &= \inf_{W \in U(\mu)} \sup_{W \cap E_N} u, && 
        \quad \text{for }\mu \in E_N.
    \end{alignat*}
    Clearly $u_N^* \le u_{N+1}^* \le u^*$.
	Suppose first $u^*(\mu)<\infty$.
	For $W \in U(\mu)$, choose a sequence $\mu_n$
    such that $\sup_W u \leq u(\mu_n) + 1/n$.
    Let $M : \N \to \N$ be a function such that $\mu_n \in E_{M(n)}$.
    Without loss of generality, we may choose $M$ to be strictly increasing.
    Thus,
    $$
    \sup_{n} \sup_{W \cap E_{M(n)}} u = \sup_W u.
    $$
    Since above holds for every $W \in U(\mu)$,
 $\lim_{N \to \infty} u_N^*(\mu) = u^*(\mu)$.
    If $u^*(\mu)=\infty$, we repeat the same argument 
    with a sequence $\mu_n$ such that $u(\mu_n)>n$.
\end{proof}

\bibliographystyle{abbrv}
\bibliography{mckean_vlasov}

\end{document}